\documentclass[11pt]{amsart}
\usepackage{latexsym}
\usepackage{amsmath,amssymb,amsfonts}
\usepackage{graphics}
\usepackage[all]{xy}

\def\Bbb{\mathbb}

\def\eea{\end{eqnarray*}}

\newtheorem{defn}{Definition}
\newtheorem{thm}{Theorem}[section]

\newtheorem{lem}[thm]{Lemma}

\begin{document}
\renewcommand{\theequation}{\thesection.\arabic{equation}}

\sloppy
\title{On the nonexistence of Einstein metric on $4$-manifolds}

\author{Chanyoung Sung}
\date{\today}

\address{Dept. of Mathematics and Institute for Mathematical Sciences \\
Konkuk University\\
         1 Hwayang-dong, Gwangjin-gu, Seoul, KOREA}
\email{cysung@kias.re.kr}
\thanks{This work was supported by the National Research Foundation(KRF) grant funded by the Korea government(MEST). (No. 2009-0074404)}

\maketitle

\begin{abstract}
By using the gluing formulae of the Seiberg-Witten invariant, we
show the nonexistence of Einstein metric on manifolds obtained from a $4$-manifold with nontrivial Seiberg-Witten invariant by performing sufficiently many connected sums or appropriate surgeries along circles or homologically trivial 2-spheres with closed oriented $4$-manifolds with negative definite intersection form.
\end{abstract}
%\vfill
%\pagebreak

\setcounter{section}{0}
\setcounter{equation}{0}

\section{Introduction}
A smooth Riemannian manifold $(M,g)$ is called Einstein if it satisfies
$$Ric_g=c g,$$ where $Ric_g$ denotes the Ricci curvature of $g$, and $c$ is a constant.
When the dimension of $M$ is less than $4$, any Einstein manifold is a space form whose classification is well-known. In higher dimensions, it is in general difficult to decide whether a manifold admits an Einstein metric. Unlike the dimension greater than $4$ where no topological obstruction is known, any closed orientable $4$-manifold $M$ admitting an Einstein metric must satisfy the Hitchin-Thorpe inequality \cite{besse, hit, tho}
$$2\chi(M)\geq 3|\tau(M)|$$ with equality held only by a quotient of $K3$ surface or $4$-torus, where $\chi(M)$ and $\tau(M)$ respectively denote the Euler characteristic and the signature of $M$. This well-known inequality is the consequence of the $4$-dimensional Chern-Gauss-Bonnet formula.

Since the $4$-dimensional geometry is complicated by the possible existence of many smooth structures, the condition for the existence of Einstein metric on $4$-manifolds inevitably involve the underlying smooth structure.
It was the Seiberg-Witten theory that has brought a remarkable improvement of the Hitchin-Thorpe condition. LeBrun exploited the curvature estimate coming from the Seiberg-Witten theory to derive that any closed oriented Einstein $4$-manifold $M$ with a monopole class satisfies
$$\chi(M)\geq  3\tau(M)$$ with equality held only by a compact complex hyperbolic 2-space or a flat 4-manifold (\cite{LE}), and

\begin{thm}[LeBrun \cite{LEB}]
Let $M$ be a smooth closed oriented $4$-manifold with a nontrivial Seiberg-Witten invariant.
Then $M\# k\overline{\Bbb CP^2}\# l(S^1\times S^3)$ does not admit Einstein metric if $k+4l> 0$ and $k+4l\geq \frac{1}{3} (2\chi(M)+ 3\tau(M)$).
\end{thm}

In this article, we generalize this theorem to :
\begin{thm}\label{th1}
Let $M$ be a smooth closed oriented $4$-manifold with a nontrivial Seiberg-Witten invariant and $N$ be a smooth closed oriented $4$-manifold with $b_2^+(N)=0$.
Then $M\# N$ does not admit Einstein metric if $$b_2(N)+4b_1(N)> 0$$ and $$b_2(N)+4b_1(N)\geq \frac{1}{3} (2\chi(M)+ 3\tau(M)).$$
\end{thm}

\begin{defn}
Let $M_1$ and $M_2$ be smooth $n$-manifolds and suppose that $k$-spheres $c_1$ and $c_2$ are embedded into $M_1$ and $M_2$ respectively with trivial normal bundle. A {\it surgery} of $M_1$ and $M_2$ along $c_i$'s are defined as the result of deleting tubular neighborhood of each $c_i$ and gluing the remainders by identifying two boundaries $S^k\times S^{n-k-1}$ using a diffeomorphism of $S^k$ and the reflection map of $S^{n-k-1}$.
\end{defn}
Note that the surgery on $M$ with $(S^1\times S^3)\# N$ along a null-homotopic circle in $M$ and a circle representing $[S^1]\times \{\textrm{pt}\}\in H_1(S^1\times S^3,\Bbb Z)$ gives $M\# N$. More generally, we will prove :

\begin{thm}\label{th2}
Let $M$ be a smooth closed oriented $4$-manifold with a nontrivial Seiberg-Witten invariant and $N_i$ be a smooth closed oriented $4$-manifold with $b_2^+(N_i)=0$ and $b_1(N_i)\geq 1$ for $i=1,\cdots, m$. Suppose that $c_i\subset N_i$ is an embedded circle nontrivial in $H_1(N_i,\Bbb R)$ for $i=1,\cdots, m$, and $\tilde{M}$ is a manifold obtained from $M$ by performing a surgery with $\cup_{i=1}^m N_i$ along $\cup_{i=1}^m c_i$.

Then $\tilde{M}$ does not admit Einstein metric if
$$\sum_{i=1}^m(b_2(N_i)+4(b_1(N_i)-1))> 0$$ and
$$\sum_{i=1}^m(b_2(N_i)+4(b_1(N_i)-1))\geq \frac{1}{3} (2\chi(M)+ 3\tau(M)).$$
\end{thm}

Most generally, we can also allow surgeries along homologically trivial 2-spheres to give :
\begin{thm}\label{th3}
Let $M$ be a smooth closed oriented $4$-manifold with a nontrivial Seiberg-Witten invariant, and $N_i, \bar{N}_j$ for $i=1,\cdots, m$ and  $j=1,\cdots, n$ be smooth closed oriented $4$-manifolds such that $b_2^+(N_i)=b_2^+(\bar{N}_j)=0$ and $b_1(N_i)\geq 1$. Suppose that $c_i\subset N_i$ for $i=1,\cdots, m$ is an embedded circle nontrivial in $H_1(N_i,\Bbb R)$, and $F_j\subset M$ and $\bar{F}_j\subset \bar{N}_j$ for $j=1,\cdots, n$ are embedded 2-spheres trivial in $H_2(M,\Bbb R)$ and $H_2(\bar{N}_j,\Bbb R)$ respectively.

If $\tilde{M}$ is a manifold obtained from $M$ by performing a surgery with $\cup_{i=1}^m N_i$ along $\cup_{i=1}^m c_i$, and with $\cup_{j=1}^n \bar{N}_j$ along $\cup_{j=1}^n F_j$ and $\cup_{j=1}^n \bar{F}_j$,
then $\tilde{M}$ does not admit Einstein metric if
$$\sum_{i=1}^m(b_2(N_i)+4(b_1(N_i)-1))+\sum_{j=1}^n(b_2(\bar{N}_j)+4(b_1(\bar{N}_j)+1))\geq \frac{1}{3} (2\chi(M)+ 3\tau(M)).$$
The same conclusion also holds when $m=0$, i.e. $\cup_{i=1}^m N_i=\emptyset$.
\end{thm}

\section{Computation of Seiberg-Witten invariant}
We will give a brief definition of the Seiberg-Witten invariant. Let $M$ be a smooth oriented Riemannian $4$-manifold and $\frak{s}$ be a Spin$^c$ structure on it.
We assume that $M$ is closed or noncompact with a cylindrical-end metric.
Let $\Bbb A(M)$ be the graded algebra over $\Bbb Z$ defined by
$$\Bbb Z[H_0(M;\Bbb Z)]\otimes \wedge^*H_1(M;\Bbb Z)$$ with
$H_0(M;\Bbb Z)$ grading two and $H_1(M;\Bbb Z)$ grading one.
An element in $\Bbb A(M)$ cannonically gives a cocycle of the Seiberg-Witten moduli space, i.e. the solution space modulo gauge transformations of the Seiberg-Witten equations of $(M,\frak{s})$. Thus the evaluation on the fundamental cycle of the moduli space is the \emph{Seiberg-Witten invariant} as a function
$$SW_{M,\frak{s}}:\Bbb A(M)\rightarrow \Bbb Z.$$ When $b_2^+(M)>1$, this is independent of a Riemannian metric and a perturbation term, thus giving a topological invariant. (If $b_2^+(M)=1$, it may depend on the chamber.) The first Chern class of a Spin$^c$ structure on $M$ whose Seiberg-Witten invariant is nontrivial is called a \emph{basic class} of $M$. For more details on the Seiberg-Witten invariant, the readers are referred to \cite{morgan, OS, sung}.

We will need the following gluing formulae of the Seiberg-Witten invariant.
\begin{lem}
Let $N$ be a closed oriented smooth $4$-manifold with negative-definite intersection form $Q$. Then there exists a Spin$^c$ structure $\frak{s}'$ on $N$ satisfying $c_{1}^{2}(\frak{s}')=-b_2(N)$.
\end{lem}
\begin{proof}
By the Donaldson's theorem, $Q$ is diagonalizable. (The original Donaldson's theorem \cite{donal} is stated for the simply-connected case, but a simple application of the Mayer-Vietoris argument gives this generalization.) Let $\{\alpha_1,\cdots,\alpha_{b_2(N)}\}$ be a basis of $H^2(N,\Bbb Z)\otimes \Bbb Q$ diagonalizing $Q$.

We have to show that there exists an element $x\in H^2(N,\Bbb Z)$ such that $Q(x,x)=-b_2(N)$, and $x$ is characteristic, i.e. $Q(x,\alpha)\equiv Q(\alpha,\alpha)$ mod 2 for any $\alpha\in H^2(N,\Bbb Z)$. This is done by taking $x=\sum_{i=1}^{b_2(N)}\pm\alpha_i$.
\end{proof}

\begin{thm}\label{mylemma}
Let $M$ and $N$ be smooth closed oriented $4$-manifolds such that
$b_2^+(M)>0$, $b_2^+(N)=0$, and $b_1(N)\geq 1$. Let $c\subset N$ be an embedded circle nontrivial in $H_1(N,\Bbb R)$ and $\tilde{M}$ be the manifold obtained by performing a surgery on $M$ with $N$ along $c$.

If $\tilde{\frak{s}}$ is the Spin$^c$
structure on $\tilde{M}$ obtained by gluing a Spin$^c$ structure $\frak{s}$ on $M$ and a Spin$^c$ structure $\frak{s}'$ on $N$ satisfying $c_{1}^{2}(\frak{s}')=-b_2(N)$,
then
$$SW_{\tilde{M},\tilde{\frak{s}}}(a\cdot [d_1]\cdots [d_{b_1(N)-1}])=\pm SW_{M,\frak{s}}(a)$$ for $a\in \Bbb A(M)$, where $[d_1],\cdots, [d_{b_1(N)-1}]$ along with $r[c]$ for some $r\in \Bbb Q$ form a basis for the non-torsion part of  $H_1(N,\Bbb Z)$.
\end{thm}
\begin{proof}
See \cite{sung}.
\end{proof}

\begin{thm}[Ozsv\'ath and Szab\'o \cite{OS}]\label{oz}
Let $M$ be a smooth closed oriented 4-manifold with $b_2^+(M)>0$. Suppose that $F\subset M$ is an embedded 2-sphere trivial in $H_2(M,\Bbb R)$, and $\tilde{M}$ is the manifold obtained by performing a surgery on $M$ with $S^4$ along $F$.

Then for each Spin$^c$ structure $\frak{s}$ on $M$, the induced Spin$^c$ structure $\tilde{\frak{s}}$ on $\tilde{M}$ satisfies
%$$\frak{s}|_{M-F}=\frak{s}'|_{M-F}$$
$$SW_{\tilde{M},\tilde{\frak{s}}}(a\cdot [\gamma])=\pm SW_{M,\frak{s}}(a)$$ for $a\in \Bbb A(M)$, where $\gamma$ is the core of surgery, i.e. a circle $\{pt\}\times D^2$ in a small tubular neighborhood $F\times D^2$ of $F$.
\end{thm}

Generalizing this, we prove :
\begin{thm}\label{newlemma}
Let $M$ and $N$ be smooth closed oriented $4$-manifolds such that
$b_2^+(M)>0$, and $b_2^+(N)=0$.
Suppose that $F\subset M$ and $\bar{F}\subset N$ are embedded 2-spheres trivial in $H_2(M,\Bbb R)$ and $H_2(N,\Bbb R)$ respectively, and $\tilde{M}$ is the manifold obtained by performing a surgery on $M$ with $N$ along $F$ and $\bar{F}$.

If $\tilde{\frak{s}}$ is the Spin$^c$
structure on $\tilde{M}$ obtained by gluing a Spin$^c$ structure $\frak{s}$ on $M$ and a Spin$^c$ structure $\frak{s}'$ on $N$ satisfying $c_{1}^{2}(\frak{s}')=-b_2(N)$,
then
$$SW_{\tilde{M},\tilde{\frak{s}}}(a\cdot [\gamma]\cdot [d_1]\cdots [d_{b_1(N)}])=\pm SW_{M,\frak{s}}(a)$$ for $a\in \Bbb A(M)$, where $\gamma$ is a circle $\{pt\}\times D^2$ in a small tubular neighborhood $F\times D^2$ of $F$, and $[d_1],\cdots, [d_{b_1(N)}]$  form a basis for the non-torsion part of  $H_1(N,\Bbb Z)$.
\end{thm}
\begin{proof}
Perform a surgery on $M$ with  $S^4$ along $F$ to obtain $M'$. In the same way, we get $N'$.  The surgery on $M'$ with $N'$ along the circle $\gamma$ gives $\tilde{M}$.

\begin{lem}
Let $\hat{M}$ be the manifold obtained from $M$ by deleting a small tubular neighborhood of  $F$. Then
$$H_1(M',\Bbb R)\simeq H_1(\hat{M},\Bbb R)\simeq H_1(M,\Bbb R)\oplus \Bbb R,$$ and $$H_2(M',\Bbb R)\simeq H_2(\hat{M},\Bbb R)\simeq H_2(M,\Bbb R),$$ where the additional $\Bbb R$-factor is generated by $[\gamma]$, and the isomorphisms are induced by the obvious inclusions. Likewise for $N'$.
\end{lem}
\begin{proof}
Obviously $H_1(M',\Bbb R)\simeq H_1(\hat{M},\Bbb R)$, because $\pi_1(M')\simeq \pi_1(\hat{M})$ by the Seifert-Van Kampen theorem.  To see $H_1(\hat{M},\Bbb R)\simeq H_1(M,\Bbb R)\oplus  \Bbb R,$ it is enough to show that
$i_*$ in the following commutative diagram of exact sequences is injective.
\[
\xymatrix{
H_{2}(\hat{M},\partial \hat{M})\ar[r]^{\partial_{*}}\ar[d]_{PD} & H_{1}(\partial \hat{M})
\ar[r]^{i_{*}}\ar[d]^{PD} & H_{1}(\hat{M})\ar[d]^{PD}\\
H^{2}(\hat{M})\ar[r]^{i^{*}} & H^{2}(\partial \hat{M})\ar[r]^{\partial^{*}} &
H^{3}(\hat{M},\partial \hat{M}).}
\]
Suppose not. Then $i^*$ in the above diagram is surjective. This means that there exists a nonzero element in $H^2(M)$, which is dual to $[F]$, yielding a contradiction. This also means that $[F]$ is zero in $H_2(\hat{M},\Bbb R)$, which will be used just below.

The fact  $H_2(\hat{M},\Bbb R)\simeq H_2(M,\Bbb R)$ follows from the exact sequence
$$H_2(\partial \hat{M})\stackrel{i_*}\rightarrow H_2(\hat{M})\oplus H_2(S^2\times D^2)\stackrel{\varphi}{\rightarrow} H_2(M)\rightarrow 0,$$ and similarly the fact $H_2(\hat{M},\Bbb R)\simeq H_2(M',\Bbb R)$ follows from the exact sequence
$$H_2(\partial \hat{M})\stackrel{i_*}\rightarrow H_2(\hat{M})\oplus H_2(D^3\times S^1)\stackrel{\varphi}{\rightarrow} H_2(M')\rightarrow 0,$$ where the sequences end with 0, because $i_* : H_{1}(\partial \hat{M})\rightarrow H_{1}( \hat{M})$ is injective.
\end{proof}

Note that $\frak{s}$ and $\frak{s}'$  restrict to be trivial on $F$ and $\bar{F}$ respectively. Thus we abuse the notation to let $\frak{s}$ and $\frak{s}'$ be the induced Spin$^c$ structures on $M'$ and $N'$ respectively.
By theorem \ref{oz},
%Ozsv\'ath and Szab\'o \cite{OS},
$$SW_{M',\frak{s}}(a\cdot [\gamma])=\pm SW_{M,\frak{s}}(a)$$ for $a\in \Bbb A(M)$. Applying  theorem \ref{mylemma},
$$SW_{M',\frak{s}}(a\cdot [\gamma])=\pm SW_{\tilde{M},\tilde{\frak{s}}}(a\cdot[\gamma]\cdot [d_1]\cdots [d_{b_1(N)}])$$ for $a\in \Bbb A(M)$.
\end{proof}

\section{Proof of Theorem  \ref{th2}}

\setcounter{equation}{0}

We need to have a basic class on $\tilde{M}$.
Let $\frak{s}$ be the Spin$^c$ structure on $M$ with a nontrivial Seiberg-Witten invariant. Applying theorem \ref{mylemma} successively, $\tilde{M}$ has nontrivial Seiberg-Witten invariant for $\tilde{\frak{s}}$. Write $c_1(\tilde{\frak{s}})$  as $c_1(\frak{s})+E$ where $E=c_1(\frak{s}')$ coming from $\cup_{i=1}^m N_i$.

Then the proof proceeds in a similar way to \cite{LEB}. First,
\begin{eqnarray*}
\chi(\tilde{M})&=&\chi(M)+\sum_{i=1}^m\chi(N_i)\\ &=& \chi(M)+\sum_{i=1}^m(2-2b_1(N_i)+b_2(N_i)),
\end{eqnarray*}
 and  $$H_2(\tilde{M},\Bbb Z)\simeq H_2(M,\Bbb Z)\oplus (\oplus_{i=1}^m H_2(N_i,\Bbb Z))$$ by a simple Mayer-Vietoris argument. (Here, we use the fact that $c_i$'s are all non-torsion.) Thus
\begin{eqnarray} \label{aaa}
2\chi(\tilde{M})+ 3\tau(\tilde{M})&=& 2\chi(M)+ 3\tau(M) \\ & & -\sum_{i=1}^m(b_2(N_i)+4(b_1(N_i)-1)).\nonumber
\end{eqnarray}
% by which we may assume that $2\chi(M)+3\tau(M)>0$, since otherwise cases are covered by %the Hitchin-Thorpe inequality.
\begin{lem}
Any Riemannian metric $g$ on $\tilde{M}$ satisfies
$$\frac{1}{4\pi^2}\int_{\tilde{M}}(\frac{s_{g}^2}{24}+2|W_+|_{g}^2)\
d\mu_{g}\geq \frac{2}{3}(2\chi(M)+3\tau(M)).$$
\end{lem}
\begin{proof}
Since $c_1(\frak{s})+ E$ and $c_1(\frak{s})- E$  are basic classes of $\tilde{M}$, LeBrun's estimate \cite{LEB} gives
\begin{eqnarray} \label{bbb}
\frac{1}{4\pi^2}\int_{\tilde{M}}(\frac{s_{g}^2}{24}+2|W_+|_{g}^2)\
d\mu_{g}\geq \frac{2}{3}((c_1(\frak{s})\pm E)^{+})^{2},
\end{eqnarray}
where $(\cdot)^+$ denotes the self-dual harmonic part. On the other hand,
\begin{eqnarray*}
((c_1(\frak{s})\pm E)^{+})^{2}&=&(c_1(\frak{s})^{+}\pm E^{+})^{2} \\
&=&(c_1(\frak{s})^{+})^{2}\pm 2c_1(\frak{s})^{+}\cdot E^{+}+(E^{+})^{2} \\
&\geq& (c_1(\frak{s})^+)^{2}\pm 2c_1(\frak{s})^{+}\cdot E^{+}.
\end{eqnarray*}
Thus at least one of $((c_1(\frak{s})+E)^{+})^{2}$ and $((c_1(\frak{s})-E)^{+})^{2}$ should be greater than or equal to $(c_1(\frak{s})^+)^2$.
Say $((c_1(\frak{s})+E)^{+})^{2}\geq  (c_1(\frak{s})^+)^2$.
Then
\begin{eqnarray*}
((c_1(\frak{s})+E)^{+})^{2}&\geq& c_1^2(\frak{s}) \\
&\geq& 2\chi(M)+3\tau(M),
\end{eqnarray*}
where we used the fact that $d(\frak{s}):=\frac{1}{4}(c_1^2(\frak{s})-(2\chi(M)+3\tau(M)))$, the dimension of the Seiberg-Witten moduli space of $(M,\frak{s})$ is nonnegative.
\end{proof}

Now suppose that $g$ is an Einstein metric on $\tilde{M}$. Then the Chern-Gauss-Bonnet formula gives :
\begin{eqnarray*}
2\chi(\tilde{M})+ 3\tau(\tilde{M})&=& \frac{1}{4\pi^2}\int_{\tilde{M}}(\frac{s_{g}^2}{24}+2|W_+|_{g}^2-\frac{|\stackrel{\circ}{r}|_g^2}{2})\
d\mu_{g} \\ &=& \frac{1}{4\pi^2}\int_{\tilde{M}}(\frac{s_{g}^2}{24}+2|W_+|_{g}^2)\
d\mu_{g} \\ &\geq& \frac{2}{3}(2\chi(M)+3\tau(M)).
\end{eqnarray*}
Combined with (\ref{aaa}), it follows that
\begin{eqnarray}\label{ddd}
\frac{1}{3}(2\chi(M)+3\tau(M))\geq \sum_{i=1}^m(b_2(N_i)+4(b_1(N_i)-1)).
\end{eqnarray}

It remains to deal with the equality case in the above inequality.
Suppose the equality holds.
Then from the above we have
\begin{eqnarray}\label{eee}
((c_1(\frak{s})+E)^{+})^{2}= c_1^2(\frak{s})= 2\chi(M)+3\tau(M).
\end{eqnarray}
Suppose $\sum_{i=1}^m(b_2(N_i)+4(b_1(N_i)-1))>0,$ which implies
 $$((c_1(\frak{s})+E)^{+})^{2}>0$$ by (\ref{ddd}) and (\ref{eee}).

 From the the equality in (\ref{bbb}), LeBrun's result \cite{LEB} says that $(\tilde{M},g)$ must be almost-K\"ahler with almost-K\"ahler form a multiple of $(c_1(\frak{s})+ E)^+$ such that the basic class $c_1(\frak{s})+ E$ being the (anti)canonical class of the associated almost-complex structure, and the almost-K\"ahler form is an eigenvector of $W_+$ everywhere.

Applying Armstrong's result \cite{arm} that any closed almost-K\"ahler Einstein 4-manifold whose almost-K\"ahler form is an eigenvector of $W_+$ everywhere is K\"ahler, or Apostolov-Armstrong-Dr\u{a}ghici's result \cite{aad} that any closed  almost-K\"ahler 4-manifold which saturates (\ref{bbb}) and whose Ricci tensor is invariant under the almost-complex structure is K\"ahler, we conclude that $(\tilde{M},g)$ is K\"ahler.

Since $(\tilde{M},g)$ is K\"ahler-Einstein, we can apply the Enriques-Kodaira classification of compact complex surfaces. Since $\tilde{M}$ has a nontrivial Seiberg-Witten invariant, its Kodaira dimension is nonnegative. Then it is minimal, because it admits a K\"ahler-Einstein metric.

Now the anti-canonical class is non-torsion, because $c_1^2(\frak{s})>0$ from (\ref{eee}). Then the basic classes of such a minimal K\"ahler surface are numerically equivalent to $rc_1(K)$, where $|r|\leq 1$ is a rational number, and $K$ is the canonical line bundle. (See \cite{morgan}.) But $\pm (c_1(\frak{s})\pm E)$ are basic classes of $\tilde{M}$. This means that $E=0$, implying that $$b_2(N_i)=0\ \ \ \forall i.$$

Finally using Wu's formula \cite{wu, HH} for a closed almost-complex 4-manifold, and  (\ref{eee}),
\begin{eqnarray*}
0&=&(c_1(\frak{s})+ E)^2-(2\chi(\tilde{M})+3\tau(\tilde{M}))\\ &=&c_1(\frak{s})^2-\sum_{i=1}^m b_2(N_i)-(2\chi(M)+ 3\tau(M)-\sum_{i=1}^m(b_2(N_i)+4(b_1(N_i)-1)))\\
&=& \sum_{i=1}^m4(b_1(N_i)-1),
\end{eqnarray*}
implying that $$b_1(N_i)=1\ \ \ \forall \ i.$$
Thus $\sum_{i=1}^m(b_2(N_i)+4(b_1(N_i)-1))=0$, yielding a contradiction.

\section{Proof of Theorem \ref{th3}}
 By successively applying theorem \ref{mylemma} and \ref{newlemma}, the Seiberg-Witten invariant of $(\tilde{M},\tilde{\frak{s}})$ is nontrivial, where $\tilde{\frak{s}}$ is the Spin$^c$ structure gotten by gluing $\frak{s}$ on $M$ which has nontrivial Seiberg-Witten invariant and $\frak{s}'$ on
$(\cup_{i=1}^m N_i)\cup (\cup_{j=1}^m \bar{N}_j)$ such that $c_1^2(\frak{s}'|_{N_i})=-b_2(N_i)$ and $c_1^2(\frak{s}'|_{\bar{N}_j})=-b_2(\bar{N}_j)$ for all $i,j$.

As before, we have
\begin{eqnarray*}
\chi(\tilde{M})&=&\chi(M)+\sum_{i=1}^m\chi(N_i)+\sum_{j=1}^n(\chi(\bar{N}_j)-4)\\ &=& \chi(M)+\sum_{i=1}^m(2-2b_1(N_i)+b_2(N_i))+\sum_{j=1}^n(-2-2b_1(\bar{N}_j)+b_2(\bar{N}_j)),
\end{eqnarray*}
 and $$H_2(\tilde{M},\Bbb R)\simeq H_2(M,\Bbb R)\oplus (\oplus_{i=1}^m H_2(N_i,\Bbb R))\oplus (\oplus_{j=1}^n H_2(\bar{N}_j,\Bbb R))$$ by a simple Mayer-Vietoris argument. (Here, we use the fact that $c_i$'s are non-torsion, and $F_j$'s and $\bar{F}_j$'s are all torsion.) Thus
\begin{eqnarray*}
2\chi(\tilde{M})+ 3\tau(\tilde{M})&=& 2\chi(M)+3\tau(M) -\sum_{i=1}^m(b_2(N_i)+4(b_1(N_i)-1))\\ & & -\sum_{j=1}^n(b_2(\bar{N}_j)+4(b_1(\bar{N}_j)+1)).
\end{eqnarray*}
Now proceeding in the same way as theorem \ref{th2},
the existence of an Einstein metric on $\tilde{M}$ dictates that
$$\frac{1}{3}(2\chi(M)+3\tau(M))\geq \sum_{i=1}^m(b_2(N_i)+4(b_1(N_i)-1))+\sum_{j=1}^n(b_2(\bar{N}_j)+4(b_1(\bar{N}_j)+1)),$$
and if the equality holds, then the left hand side of the above inequality is positive, and the same argument as theorem \ref{th2} gives that $$b_2(N_i)=b_2(\tilde{N}_j)=0\ \ \ \forall i,j,$$ and
$$\sum_{i=1}^m4(b_1(N_i)-1))+\sum_{j=1}^n4(b_1(\bar{N}_j)+1)=0$$ which is a contradiction.

\section{Final Remarks}

Unlike the surgery with $N$ with $b_2^+(N)=0$, in case of a surgery with 4-manifolds with $b_2^+>0$ it is difficult to decide the existence of Einstein metric, because those surgered manifolds have no basic classes and it is very difficult to show the existence of Seiberg-Witten equations for a general metric.

Ishida and LeBrun  used the Bauer-Furuta invariant \cite{BF, bau} whose nonvanishing also guarantees the existence of Seiberg-Witten equations for any metric %refining the Seiberg-Witten invariant
to show the nonexistence of Einstein metric on some connected sums of K\"ahler surfaces.
As in \cite{IL}, let $X_j$ for $j=1,\cdots,4$ be smooth closed almost-complex 4-manifolds satisfying $$b_1(X_j)=0,\ \ b_2^+(X_j)\equiv 3\  \textrm{mod}\ 4,\ \  \sum_{j=1}^4b_2^+(X_j)\equiv 4\ \textrm{mod}\ 8,$$ and $N$ be any smooth closed oriented 4-manifold with $b_2^+(N)=0$. Suppose that all $X_i$'s have nonzero mod-2 Seiberg-Witten invariants. Then, for each $m=2,3,4$, $$(\#_{j=1}^m X_j)\# N$$ does not admit Einstein metric if $$4m-(2\chi(N)+3\tau(N))\geq \frac{1}{3}\sum^m_{j=1}c_1^2(X_j).$$

%A further exploration for other connected sums K\"ahler surfaces will be continued
Finally one can use the $G$-monopole invariant \cite{ruan} which is roughly the ``count" of $G$-invariant solutions of the Seiberg-Witten equations modulo gauge transformation to  show the nonexistence of $G$-invariant Einstein metric on some 4-manifolds with a $G$-action for a compact Lie group $G$.

\bigskip

%\noindent{\bf Acknowledgement.} The author warmly thanks.

\end{document}